\documentclass[12pt,a4paper]{amsart}
\usepackage{geometry, comment}
\usepackage{color}
\usepackage{graphicx}
\usepackage{amssymb}
\usepackage{amsrefs}

\newcommand{\N}{\mathbb{N}}
\newcommand{\R}{\mathbb{R}}
\newcommand{\Z}{\mathbb{Z}}
\renewcommand{\l}{\lambda}
\newcommand{\Haus}{\mathcal{H}}     % Misura di Hausdorff

\newcommand{\B}{\mathcal{B}}

\newcommand{\T}{\mathbb{T}}

\renewcommand{\a}{\alpha}
\renewcommand{\b}{\beta}
\renewcommand{\d}{\delta}

\newcommand{\eps}{\varepsilon}

\renewcommand{\l}{\lambda}

\newcommand{\s}{\sigma}

\newcommand{\proj}{\operatorname{proj}}
\newcommand{\oo}{\operatorname{o}}
\newcommand{\OO}{\operatorname{O}}
\newcommand{\dv}{\operatorname{div}}
\newcommand{\Int}{\operatorname{int}}

\newcommand{\co}{\hbox{\rm co\,}}
\newcommand{\txt}{\qquad\hbox}

\renewcommand{\O}{\Omega}

\renewcommand{\i}{\infty}
\newcommand{\p}{\partial}
\newcommand{\ov}{\overline}

\newcommand{\limsups}{\operatorname{lim\,sup^\#}}
\newcommand{\liminfs}{\operatorname{lim\,inf_\#}}
\newtheorem{theorem}{Theorem}[section]

\newtheorem{proposition}[theorem]{Proposition}
\newtheorem{lemma}[theorem]{Lemma}
\theoremstyle{remark}\newtheorem{remark}[theorem]{Remark}
\theoremstyle{definition}\newtheorem{definition}[theorem]{Definition}

\theoremstyle{definition}

\theoremstyle{remark}
\renewcommand{\l}{\lambda}

\title{ Homogenization of the G--Equation: a metric approach }
\author[Siconolfi]{Antonio Siconolfi}

\address{Department of Mathematics \\
                 Universit\`a degli Studi di Roma ``La Sapienza''\\
         Piazzale Aldo Moro 5  \\ 00185 Roma\\
                   Italy.}
\email{siconolf@mat.uniroma1.it}


\begin{document}
\maketitle

\begin{abstract}
The aim of the paper is to  recover some results of  Cardaliaguet--Nolen--Souganidis in \cite{CNS} and  Xin-- Yu in \cite{XY}  about the homogenization of the G--equation, using different and simpler techniques. The main mathematical issue is the lack of coercivity of the Hamiltonians. In our approach we consider a multivalued dynamics  without periodic invariants sets, a family of intrinsic distances and perform an approximation by a sequence of coercive Hamiltonians.
\end{abstract}

\section{Introduction}

True to the title, the aim of these notes is to study the homogenization of time dependent Hamilton--Jacobi equations with Hamiltonian
\[H(x,p)= |p| + p \cdot V(x) \qquad (x,p) \in \R^N \times \R^N,\]
where the advection field $V$ is $\Z^N$--periodic, Lipschitz continuous  and with suitably small divergence. This Hamiltonian appears in combustion theory, where  the homogenization result can be interestingly interpreted, see \cite{CNS}, \cite{XY}. However, we do not treat this side of the issue, and stay solely concerned with the theoretical aspects of the matter.

As it is well known, the relevant mathematical point is that $H$ looses coercivity whenever $|V| \geq 1$ and consequently, as pointed out
in \cite{CNS}, \cite{XY},  there could be no homogenization if $V$ is just assumed to be Lipschitz continuous or even more regular, some additional requirements are needed, specifically the condition on the divergence. This is in contrast to the homogenization theory for coercive Hamiltonians, see \cite{LPV}.

It is opportune to make clear that our outputs are not new, in \cite{CNS} it has been  already established, among other things,  that homogenization takes place under assumptions just slightly different  from ours, and for vector fields $V$ also depending on time, which adds further  complications.

The interest of our contribution, if any, is that the  proofs are different and, we believe, more simple  and intuitive. They rely on a metric approach  to the matter.   Note that metric techniques have been already used in homogenization of Hamilton--Jacobi equations, in particular in the stationary ergodic case, see  for instance \cite{AS}, \cite{DS}.

 To begin with,  we   define the effective Hamiltonian $\ov H$, appearing in the limit equation, as
\[ \ov H(P)= \inf\{a \mid  H(x,Du+P) =a \;\hbox{ admits an
usc bdd periodic subsolution}\},\]
without passing through ergodic  approximation. From this formulation, we see  that the subsolutions of  $H(x,Du+P) = \ov H(P) + \d$, for $\d >0$ small,  can be used as approximate correctors to prove, via  the  perturbed test function method, see \cite{E},     that the weak lower semilimit of the solutions to the $\eps$--oscillatory equations is   supersolution of the limit equation with the effective Hamiltonian. This depends on the very definition of $\ov H$, and  can be done without using the assumption on the divergence of $V$.

The  solution of the problem therefore boils down  to show the  existence of lsc bounded periodic supersolutions to the critical equation  $H(x,Du+P) = \ov H(P)$. This  should actually allow performing the other half of the homogenization procedure.

Backtracking  the issue, we see that such  supersolutions do exist provided that some periodic distances related to $H$ are bounded. In other terms we discover, not surprisingly,  that  some controllability condition is needed in order for homogenization to take place, compare for instance with \cite{AB}, where other noncoercive  models are considered.

The boundedness  of the distances is in turn equivalent, see Sections \ref{distone}--6, to the nonexistence of periodic invariant sets for a multivalued  dynamics related to $H$. This is the key  point where  the condition on the divergence of $V$ enters into play.

The key point in this respect is  Theorem \ref{invariant}, where it is proved that if such a  set, say $\Omega$, does exist, then it  possesses Lipschitz boundary, but then the compression of $V$ toward the interior of $\Omega$ is not  compatible with its small divergence. This result is deduced through divergence theorem and isoperimetric inequality, see Lemma \ref{perimeter}, Theorem \ref{divergence}. More precisely, we apply the isoperimetric inequality on the flat torus $\T^N = \R^N/\Z^N$, so that our condition on $V$ reads
\[ \|\dv V\|_{L^N(\T^N)} \leq \frac 1\chi,\]
where $\chi$ is the isoperimetric constant on $\T^N$, which coincides with that of $Q_{\frac 12}= (0,1/2)^N$, see Section \ref{structure}. This is actually  the same assumption of \cite{CNS} except that they use the isoperimetric constant of $Q_1=(0,1)^N$. We are not aware of comparison results between these two constants, and we do not know if this difference is just a technical detail, or some more deep fact is involved.

The paper is organized as follows: in Section 2 we give notations and preliminary results. In Section 3 we write down the problem with the standing assumptions. Section 4 is devoted to the definition of a relevant multivalued dynamics and the proof of  the nonexistence of periodic invariant sets. In section 5 we define some distances intrinsically related to $H$ and study their properties.   In Section 5 we approximate $H$ by a sequence $H_k$ of coercive Hamiltonians and prove that the effective Hamiltonians $\ov H_k$ converge to $\ov H$. From this result we deduce the existence of lsc periodic bounded supersolutions to the critical equation related to $H$, and complete the proof of the homogenization result. Finally, in the Appendix we present some proofs.

\bigskip

\section{Preliminaries}\label{structure}

Given a subset $E$ of $\R^N$, we write  $\ov E$, $\Int E$, $\partial E$, $E^c$ to indicate its closure, interior, boundary, complement, respectively. For $x \in \R^N$, $r >0$, we denote by $\B(x,r)$ the open ball centered at $x$ with radius $r$.  Given $a > 0$, we denote by $Q_a$ the open hypercube $(0,a)^N$. We write $\T^N= \R^N/\Z^N$ to indicate the flat torus. Namely the torus with the metric induced by the Euclidean one on $\R^N$. With  this choice $\R^N$ and $\T^N$ are locally isometric.

We  identify $\R^N$ and its dual, and use the symbol  $\cdot$ to denote both  the duality pairing and the scalar product.

Given a convex subset $C$ of $\R^N$, the negative polar cone $C^-$ is defined through
\[C^-= \{p \in \R^N \mid p \cdot q \leq 0 \;\hbox{for any  $ q \in C$}\}.\]

All the  curves we will consider  are assumed to be  Lipschitz continuous.

 The acronyms  usc /lsc stand for upper/lower semicontinuous.
For any bounded function $u: \R^N \to \R$, and $\d >0$  we define $\d$--inf and sup convolutions, respectively, as follows:
\begin{eqnarray*}
  u_\d (x)&=& \inf \left \{u(y) + \frac 1{2\d} \, |x-y|^2 \mid y \in \R^N \right \} \\
  u^\d(x) &=& \sup \left \{u(y) - \frac 1{2\d} \, |x-y|^2 \mid y \in \R^N \right \}
\end{eqnarray*}
Given $x_0 \in \R^N$, we say that $y_0$ is $u_\d$--optimal with respect to $x_0$  if $y_0$ realizes the infimum in the formula of inf convolution. The notion of $u^\d$-- optimal point is given similarly.

Given a sequence of locally bounded functions $u_n$, we define the lower/upper weak semilimit as follows
\begin{eqnarray*}
  \liminfs u_n(x) &=& \inf \{ \liminf\{u_n(x_n) \mid x_n \to x\} \\
  \limsups u_n(x) &=& \sup \{ \limsup\{u_n(x_n) \mid x_n \to x\}
\end{eqnarray*}

The term super/sub solution for a given   PDE  must be understood in the viscosity sense. We denote by $D^+u(x)$ (resp. $D^- u(x)$) the viscosity superdifferential (resp. subdifferential).
\smallskip

\begin{definition}\label{lipp} We say that a subset $\O \subset \R^N$ has
{\em Lipschitz boundary} if for any $x \in \p \O$  there is  a
neighborhood $U$, an hyperplane  of the form $x + p^\perp$,
for some nonvanishing vector $p$,  and a Lipschitz function $\psi$
defined in $(x + p^\perp) \cap U$ with
\begin{itemize}
    \item[{\bf (i)}] $  \p\O \cap U=\{ y + \psi(y) \, p \mid y \in (x + p^\perp) \cap U\}$
    \item[{\bf (ii)}]  $ \Int\O \cap U = \{ y + \l \, p \mid y \in (x + p^\perp) \cap U, \, \l
    > \psi(y)\}$
\end{itemize}
\end{definition}

\smallskip

The isoperimetric constant $\chi$  on the torus $\T^N$ is the smallest  constant  such that
\[|\Theta|^{1-1/N} \leq \chi\, \Haus^{N-1}(\partial \Theta)\]
for any subset $\Theta$ of  $\T^N$ with Lipschitz boundary. Here $\Haus^{N-1}$ stands for the $(N-1)$--dimensional Hausdorff measure and $|\cdot|$ for the $N$--dimensional Lebesgue measure.

 The (relative) isoperimetric  constant $\chi_a$ on $Q_a$ is the smallest constant such that
 \[|\Theta|^{1-1/N} \leq \chi_a\, \Haus^{N-1}(\partial \Theta \cap Q_a)\]
 for all subset $\Theta \subset Q_a$ with Lipschitz boundary.

 According to \cite{R},  the isoperimetric problems in $\T^N$   in $Q_{\frac 12}$ are equivalent, and consequently the constants $\chi$ and $\chi_{\frac 12}$ coincide.

\smallskip

A formulation of the  divergence  Theorem on the torus is:

Let $\Theta \subset \T^N$ be  a set with Lipschitz boundary and $V: \T^N \to \R^N$ a Lipschitz continuous vector field, then
\[ \int_\Theta \dv V \, dx= \int_{\partial \Theta} V \cdot n_\Theta \,
d\,\Haus^{N-1},\] where $n_\Theta$ is the outward unit normal on $\partial\Theta$.

\smallskip

\begin{remark}\label{normale} We recall that if $\Theta$   has Lipschitz boundary then the outer normal $n_\Theta(x)$  exists for $\Haus^{N-1}$ a.e.  $ x \in \partial \Theta$ and satisfies the following property:
\[\proj_\Theta(x+ t \,n_\Theta(x))= x \txt{for $t >0$ suitably small,}\]
where $\proj_\Theta$ indicates the projection on $\ov\Theta$.
\end{remark}

\bigskip
\section{The problem}

We consider the Hamiltonian
\[H(x,p)= |p| + p \cdot V(x)   \txt{in $\R^N \times \R^N$,}\]
where $V: \R^N \to \R^N$, is a vector field  that we assume
\begin{itemize}
  \item[{\bf (A1)}] $\Z^N$--periodic and Lipschitz continuous;
  \item[{\bf (A2)}] satisfying $\|\dv V\|_{L^N(\T^N)} \leq \frac 1 \chi$, where $\chi$ is the isoperimetric constant in $Q_{\frac 12}$ and in $\T^N$ as well, see Section \ref{structure}.
\end{itemize}

\medskip

Throughout the paper  we will denote by $L_V$, $M_V$, the Lipschitz constant of $V$  and the maximum of $|V(x)|$, respectively.  We consider, for $\eps > 0$,  the family of time--dependent Hamilton--Jacobi equations

\begin{equation}\label{HJe} \tag{HJ$_\eps$}
\left\{\begin{array}{cc}
         (u_\eps)_t(x/\eps,t) +H(x/\eps,Du_\eps) & =0 \qquad\hbox{in $\R^N \times (0,+\infty)$}  \\
         u_\eps(\cdot,0) & = u_0  \qquad\hbox{in $\R^N$}\\
       \end{array} \right .
\end{equation}
where

\begin{itemize}
  \item[{\bf (A3)}] $u_0$ is bounded uniformly continuous.
\end{itemize}

Our goal is to study the asymptotic behavior of the (unique) solutions $u^\eps$ of \eqref{HJe}, as $\eps$ goes to $0$, and to prove that they locally uniformly converges to a function $u^0$ which is (unique) solution of a problem of the form

\begin{equation}\label{HJ} \tag{$ \overline {\mathrm{HJ}}$}
\left\{\begin{array}{cc}
         u_t(x,t) +\overline H(Du) & =0 \qquad\hbox{in $\R^N \times (0,+\infty)$}  \\
         u(\cdot,0) & = u_0  \qquad\hbox{in $\R^N$}\\
       \end{array} \right .
\end{equation}
where $\ov H$ is a suitable limit Hamiltonian  convex and positively homogeneous.

Conditions  ${\bf (A_1)}$, ${\bf (A_2)}$, ${\bf (A_3)}$ will be assumed throughout the paper,  without any further mentioning.

\bigskip

\section{Invariant sets}

We introduce a set--valued vector field related to $V$. We set for any $x \in \R^N$,
\[ F(x) = F_1(x)=\ov\co\{ \B(V(x),1) \cup \{0\}\},\]
where $\ov\co$ stands for closed convex hull. We also define for $\d \in (0,1)$
\[F_\d(x) = \ov\co\{ \B(V(x),\d) \cup \{0\}\}.\]

\medskip

The set--valued vector fields $F_\d(x)$ are
Hausdorff--continuous with compact convex values for any $\d \in
(0,1]$. We will say that a  curve $\xi:[0,T] \to \R^N$
is an {\em integral trajectory} of $F$ if
\[ \dot\xi(t) \in F(\xi(t)) \qquad\hbox{for a.e. $t \in [0,T]$.}\]

\smallskip

\begin{definition}\label{definva} We say that a set $\O \subset \R^N$ is
{\em invariant} for $F$ if $\O$ is a proper subset of $\R^N$ and for
any integral curve $\xi$ of $F$ defined in some interval $[0,t]$
with $\xi(0) \in \O$ (resp. $\xi(t) \in \O$), we have
\[ \xi(s) \in \O \txt{for any $s \in [0,t]$.}\]
\end{definition}

\smallskip

Our first result is:

\begin{theorem}\label{invariant}  Any   invariant
set for  the set--valued  dynamic $\dot\xi \in F(\xi)$ in $\R^N$ has
Lipschitz boundary.\end{theorem}

\medskip

\begin{remark} The same statement holds clearly true for the dynamics given by $\dot \xi \in \B(V(\xi),1)$, since the integral curves of the two problems are the same, up to change of parameter. The choice of $F$ fits our analysis better.

\end{remark}

\medskip

\begin{remark}\label{icecream} The above statement is nontrivial only if $H$ is noncoercive, or in other terms if $|V(x)| \geq 1$ for some $x$.  If on the contrary $|V(x)| < 1$ for any $x$ then $0 \in \Int F(x)$ and any curve of $\R^N$ is an integral trajectory of $F$, up to change of parameter, so that invariant subsets cannot  exist. Actually, we reach in our setting the conclusion  that no periodic invariant set exist, but some mathematical effort is needed.
\end{remark}

\medskip

We need some preliminary material. The next assertion is easy to check.

\begin{lemma}\label{corkey} Given  $\d_2 > \d_1$  in $(0,1]$
\[  F_{\d_1}(x) \subset  F_{\d_2}(y) \txt{for any $x \in \R^N$, $y \in \B(x,r_V(\d_1,\d_2))$,}\]
with $r_V(\d_1,\d_2)= \frac {\d_2-\d_1}{L_V}$. CORRECTION
\end{lemma}

\medskip

\begin{lemma}\label{aggiunto} Let $\O$ be an   invariant set for $F$, then
\[|V(x)| \geq 1 \qquad\hbox{for any $x_0 \in \partial \Omega$.}\]
\end{lemma}
\begin{proof} Assume for purposes of contradiction that there is $x_0 \in \partial \O$ with $|V(x_0)| < 1$, then we find  $r >0$ such that the same strict inequality holds true in $\B(x_0,r)$.  This implies that
\[ 0 \in \Int F(x)\qquad\hbox{for any $x \in \B(x_0,r)$.}\]
Any curve lying in $\B(x_0,r)$ is therefore an integral curve of $F$, up to change of parameter. This is true in particular for a segment connecting a point in $\B(x_0,r) \cap \O$ to another point in $\B(x_0,r) \setminus \O$, which is in contrast with the invariance of $\O$.
\end{proof}

\medskip

\begin{lemma}\label{lem2inva} Let $\O$, $x$, $\d$  be an  invariant set for $F$, a point
 in $\R^N$ and a constant in $(0,1)$,
then
\begin{eqnarray*}
 (y + \Int F_\d(x)) \cap \B  (x, r_V(\d,1)) &\subset&  \Int\Omega \;\txt{ for any $y \in \ov\Omega
\cap U$}\\
(y - \Int F_\d(x)) \cap \B  (x, r_V(\d,1)) &\subset&   \Int\Omega^c \txt{ for any $y
\in \ov{\Omega^c} \cap U$.}
\end{eqnarray*}
\end{lemma}
\begin{proof} To ease notations, we set $U =\B  (x, r_V(\d,1))$.  First take $y \in \Omega  \cap  U$ and $q \in F_\d(x)$ with $y + q \in
U$, then the curve $y + t \, q$, $t \in [0,1]$, is an integral curve
of $F$ in force of Lemma \ref{corkey}. This shows
\[ (y + F_\d(x)) \cap  U \subset \O\]
and consequently
\begin{equation}\label{inva1}
    \Int \big ((y + F_\d(x)) \cap U) = (y + \Int F_\d(x)) \cap U
\subset \Int \O \qquad\hbox{for any $y \in \O \cap  U$.}
\end{equation}
 Now consider  $y \in \ov\Omega \cap U$ and denote by $y_n$ a sequence in $\Omega \cap
U$ converging to $y$. If $q \in \Int F_\d(x)$ with $y + q \in U$, then
\[ q + y-y_n \in \Int F_\d(x) \txt{for $n$ sufficiently large}\]
and consequently
\[y+q = y_n + (q +y-y_n) \in (y_n + \Int F_\d(x)) \cap U. \]
This proves
\[ (y + \Int F_\d(x)) \cap U \subset  \bigcup_n \, (y_n + \Int F_\d(x)) \cap U,\]
since $y_n \in \Omega \cap U$ for any $n$,  all the sets in the
right hand--side of the previous formula are contained in $\Int
\Omega$ by \eqref{inva1}.  The first inclusion of the claim is then
proved. The latter can be deduced by slightly adapting the same
argument.
\end{proof}

\medskip

\noindent \begin{proof} [ Proof of Theorem \ref{invariant}]
Let $\Omega$ be an invariant set. Given $x \in \p \O$, we select $\d \in (0,1)$ and
set $U = \B(x,r_V(\d,1))$. We denote by $\proj$  the orthogonal projection on the space $x +
V(x)^\perp$ and consider the set
\[
    I:=\proj((x + \Int F_\d(x))\cap U) \cap
\proj((x- \Int F_\d(x)) \cap U),
\]
 which is  an open convex neighborhood of $x$.
\medskip

{\bf Claim 1.} \; {\em For any $y \in I$ there is one and only one
$\l=\l(y) \in \R$ such that $y + \l \,V(x) \in \partial \Omega \cap
U$}.

\smallskip
Let $y$ be in $I$, by the very definition of $I$ there are $\l_1$,
$\l_2$ with
\begin{eqnarray*}
  y + \l_2\,V(x) &\in&  (x+ \Int F_\d(x)) \cap U \\
  y + \l_1\,V(x) &\in&  (x- \Int F_\d(x)) \cap U
\end{eqnarray*}
note that  by Lemma \ref{aggiunto} $|V(x)| \geq 1$, then   $V(x)^\perp$ is a supporting hyperplane for both $x+  F_\d(x)$ and $x -  F_\d(x)$, and
\begin{eqnarray*}
  x+ \Int F_\d(x) &\subset& \{z \mid (z-x) \cdot V(x) > 0\} \\
  x - \Int  F_\d(x) &\subset& \{z \mid (z-x) \cdot V(x) < 0\}.
\end{eqnarray*}
This implies that  $\l_2 >0$ and $\l_1 <0$.  In addition, owing  to
the convexity of $U$ we have
\begin{equation}\label{inva3}
    y + \l \, V(x) \in U \txt{for any $\l \in [\l_1,\l_2]$.}
\end{equation}
 We deduce from Lemma \ref{lem2inva}
\begin{eqnarray*}
  y + \l_2\,V(x) &\in&  \Int \O \\
  y + \l_1\,V(x) &\in&  \Int \O^c.
\end{eqnarray*}
 We set
\begin{eqnarray*}
\l' &=& \inf \{ \l \in (\l_1,\l_2) \mid y + \l\,V(x) \in
\Int\Omega\} \\
\l'' &=& \sup \{ \l \in (\l_1,\l_2) \mid y + \l\,V(x) \in
\Int\Omega^c\}.
\end{eqnarray*}
We proceed proving by contradiction that $\l'= \l''$. In fact,  if $\l' >
\l''$ then
 \[ y + \l\,V(x) \in \partial\Omega \cap U \txt{for any $ \l \in [\l'',\l']$},\]
on the other side, we choose $\l$ in $[\l'', \l']$ so close to
$\l''$ that
\[(\l-\l'') \, V(x) \in \Int F_\d(x), \]
 and we deduce from  \eqref{inva3} and Lemma \ref{lem2inva}
 \begin{eqnarray*}
   y + \l \, V(x) &= &y + \l'' \, V(x) +(\l - \l'')\, V(x) \\ &\in& ((y+
\l'' \,V(x)) + \Int F_\d(x)) \cap U
    \subset \Int  \Omega,
 \end{eqnarray*}
  which is
impossible. If instead  $\l' < \l''$, then we set
\[\l_0 = \sup\{ \l \in [\l',\l''] \mid y + \l\,V(x) \in
\Int\Omega\}.
\]
 We find $\l< \l_0$ close to $\l_0$ with
\begin{eqnarray*}
  y + \l\,V(x) &\in& \Int\Omega \\
  (\l_0 - \l) \, V(x)  &\in& \Int F_\d(x)
\end{eqnarray*}
and
\begin{eqnarray*}
  y + \l_0\,V(x) &=& y + \l \, V(x) +(\l_0 - \l)\, V(x) \\
   &\in& ((y+
\l \,V(x)) + \Int F_\d(x)) \cap U \subset \Int  \Omega ,
\end{eqnarray*}
contrast with the definition of $\l_0 $. We have in conclusion
established by contradiction that $\l'=\l''$.

\medskip

{\bf Claim 2.} {\em  The function
\[ y \mapsto \{ \l \mid y+ \l\,V(x) \in \partial\Omega\}\]
 from $I$ to $\R$ is Lipschitz--continuous}.

The function, that we denote by $\psi$,  is univocally defined by the
previous claim, and  is continuous by its very definition and the
fact that $\partial\Omega$ is closed.  By squeezing a bit $I$, we
can in addition suppose without loosing generality that $\psi$ is
uniformly continuous  so that we can determine $\eps$ with
\[|y_1-y_2| < \eps \, \Rightarrow \, |\psi(y_1) - \psi(y_2)| < 1.\]
We pick such a pair $y_1$, $y_2$, and  assume that   $\psi(y_1) >
\psi(y_2)$, therefore
\[\B( (\psi(y_1) - \psi(y_2))\,V(x),( \psi(y_1) - \psi(y_2))\, \d ) \
\subset \Int F_\d(x)\] and consequently by Lemma \ref{lem2inva}
\[ ((y_2 + \psi(y_2) \, V(x)) + \B( (\psi(y_1) - \psi(y_2))\,V(x),( \psi(y_1) - \psi(y_2))\, \d
)) \cap U \subset \Int \O.\] Since $y_1 + \psi(y_1) \, V(x) \in \p
\O \cap U$, we must therefore have
\begin{eqnarray*}
 |y_2 -y_1| &=&|y_1 + \psi(y_1) \, V(x) - y_2 - \psi(y_2) \, V(x) - (\psi(y_1) - \psi(y_2))\,V(x)|  \\
   &\geq& ( \psi(y_1) - \psi(y_2))
\d
\end{eqnarray*}
This shows the claim. To conclude, it is enough to notice, see Definition \ref{lipp}, that by Lemma \ref{lem2inva} the
set
\[ \{ y + \l \, \psi(y) \mid  y \in I, \, \l > \psi(y)\}\]
is contained in $\Int \O$.
\end{proof}

\medskip
\begin{definition} We say that a subset $A \subset \R^N$ is { \em periodic}
 if
\[ x \in A \Rightarrow  x +  z \in A \txt{for any $z \in
\Z^N$.}\]
\end{definition}

\smallskip
The main consequence of the  Theorem \ref{invariant} is:

\begin{theorem}\label{divergence}
The dynamics given by  $F$  does not have any invariant periodic subset.
\end{theorem}

\smallskip
\begin{remark}  Note that invariant non periodic sets for $F$ could exist in our setting. It is enough to consider for instance $V(x)$ constant.
\end{remark}

\smallskip
We preliminarily give  two lemmata. We define for $x \in \R^N$
\[Z^\i(x)= \{p \mid H(x,p) \leq 0\},\]
It is apparent that $Z^\i(x)$ is the recession  cone of the sublevels $\{p \mid H(x,p) \leq a\}$ for $a >0$.

\begin{lemma}\label{corco} \hfill
\[ F(x)^-=Z^\i(x) \txt{for any $x \in \R^N$.}\]
\end{lemma}

\begin{proof}  An element of $F(x)$ can be written in the form
\[\l \, (V(x) + q) \txt{for $0 \leq \l \leq 1$, $|q| \leq 1$.}\]
Let $p \in Z^\i(x)$, then
\[p \cdot (V(x) + q) \leq - |p| + p \cdot q \leq 0,\]
which shows that
\[Z^\i(x) \subset F(x)^-.\]
Conversely, let $p \in F(x)^-$, then $V(x) + \frac p{|p|} \in F(x)$,
and so
\[ 0 \geq p \cdot \left ( V(x) + \frac p{|p|} \right ) = H(x,p),\]
which implies that $p \in Z^\i(x)$. This concludes the proof.
\end{proof}

\medskip

\begin{lemma}\label{perimeter}  Let $\Omega \subset \R^N$ be an  invariant set for $F$. Then
\begin{equation}\label{perimeter1}
     n_\O(x) \cdot V(x) \leq  -1  \txt{for $\Haus^{N-1}$  a.e. $x \in
\p\Omega$,}
\end{equation}
where  $n_\O$ denotes the outer unit normal to $ \Omega$ and
$\Haus^{N-1}$ is the $(N-1)$--dimensional Hausdorff measure.
\end{lemma}
\begin{proof}  Recall that $\Omega$ has
Lipschitz boundary  by Theorem \ref{invariant}.
 Let $x$ be a point of $\p\O$ where there exists the  outer normal $n_\O(x)$, see Remark \ref{normale}. We recall
that by Lemma \ref{aggiunto} $|V(x)| \geq 1 $.   By Lemma \ref{lem2inva}, we can find for
any $\d \in (0,1)$ an open ball $U$ centered at $x$ with
\[(x + \Int F_\d(x)) \cap U \subset  \Int\Omega, \]
  which implies  that $n_\O(x)$ is normal to $x + F_\d(x)$ at $x$
  and consequently $n_\O(x) \in F_\d(x)^-$.
Therefore, by Lemma \ref{corco}
\[n_\O(x) \in \bigcap_{\d \in (0,1)}  \, F_\d(x)^- = F(x)^-=Z^\i(x), \]
or, in other terms
\[0 \geq H(x,n_\O(x)) \geq 1 + n_\O(x) \cdot V(x).\]
This proves \eqref{perimeter1}.

\end{proof}

\medskip

\noindent  \begin{proof}[Proof of Theorem \ref{divergence} ] Assume for purposes of contradiction that there
exists a periodic  invariant set for $F$. We denote by $\Theta$ its
projection on $\T^N$,  by Theorem \ref{invariant} $\Theta$ possess
Lipschitz boundary, so that we are  in position to apply
divergence Theorem and, owing to  Lemma \ref{perimeter}, we
get
\[\Haus^{N-1}(\partial\Theta) \leq \int_{\partial\Theta} -  n_\Theta \cdot V
\,d\,\Haus^{N-1}= \left | \int_\Theta \dv V \, dx \right
|\leq \|\dv\, V\|_{L^N(\T^N)} \, |\Theta|^{1 - 1/N}.\] Exploiting
assumption {\bf (A2)} on $\dv V$ we deduce
\[ \Haus^{N-1}(\partial\Theta) < \frac 1\chi \, |\Theta|^{1 - 1/N},\]
which is contrast with the isoperimetric inequality, see Section \ref{structure}.
\end{proof}

\section{ Intrinsic distances}\label{distone}

We associate to the $1$--sublevels of $H$   an intrinsic distance. The first step is to define for any $x$, $q$ in $\R^N$
 \begin{eqnarray*}
Z(x) &=& \{p \mid H(x,p) \leq 1 \} \\
\s(x,q) &=& \sup \{p \cdot q \mid p \in Z(x)\} \ \end{eqnarray*}
 The set--valued function $Z$ is closed
convex valued and   $0$ is an interior point of $Z(x)$ for any $x$. This
implies that  $\s(x,q)$ is strictly positive when $q \neq 0$.

For $|V(x)|$ increasing,  the sublevel $Z(x)$  shrinks
in the direction of $V(x)$ and stretches  in the opposite one,
assuming more or less the shape of a water drop, till $|V(x)|$
reaches the threshold value of $1$, where it breaks apart and the unbounded recession cone $Z^\infty(x)$ pops up as a subset.
Accordingly,  the support function $\s(x,\cdot)$  becomes infinite
at some $q$.

The next result is straightforward, we provide a proof in the Appendix for completeness.

\medskip

\begin{proposition}\label{sigmalsc}  The function $(x,q) \mapsto \s(x,q)$  from $\R^n \times \R^N$ to
 $ \R^+ \cup \{+ \infty\}$ is lower
semicontinuous. In addition $q \mapsto \s(x,q)$ is convex
positively homogeneous, for any fixed $x$.
\end{proposition}

\smallskip

The next result puts in relation the present construction and $F$.

\smallskip
\begin{lemma}\label{co}     We have
\[ \s(x,q) \leq 1 \txt{for any $x \in \R^N$, $q \in F(x)$.}\]
\end{lemma}
\begin{proof}
Given
\begin{equation}\label{co2}
   q  = V(x) + \bar q  \txt{with $\bar q$ unit vector},
\end{equation}
 $p \in Z(x)$, we
have
\[p \cdot q = p \cdot V(x) + p \cdot \bar q \leq p \cdot V(x) + |p|
\leq 1,\] which shows that $ \s(x,q) \leq 1$. Since $F(x)$ is the convex
hull of vectors of the form \eqref{co2} plus $0$, we get the assertion.

\end{proof}

\medskip

For any  curve $\xi:[0,T] \to \R^N$ we define
the {\em intrinsic length functional} $\ell_V$ by
\[\ell_V(\xi) = \int_0^T \s(\xi,\dot\xi)\, dt. \]
Notice that
 the above integral is invariant for orientation--preserving
change of parameter, as an intrinsic length should be. It is in addition positive, and can be infinite.   The
corresponding {\em path--metric} $S_1$ is given by
\begin{equation}\label{metro0}
 S_1(x,y)= \inf\{\ell_V(\xi) \mid \;\hbox{$\xi$ joining $x$ to
$y$}\},
\end{equation}
for any $x$, $y$ in $\R^N$. The (semi)distance $S_1$ is
strictly positive whenever the two arguments are distinct, and can be infinite. Note that $S_1(x,x)=0$ for any $x$, in addition $S_1$   enjoys the triangle property, but it is not in general
symmetric.
\medskip

We provide a periodic version of $S_1$ via the formula
\begin{equation}\label{metro1}
  \ov S_1(x,y) = \inf \{S_1(x+z,y + z') \mid z, \, z' \in \Z^N\},
\end{equation}
Due to the periodic character of $H$, we also have
\[\ov S_1(x,y)=\inf \{S_1(x+z,y ) \mid z  \in \Z^N\} = \inf \{S_1(x,y + z) \mid z  \in
\Z^N\}.\]
\smallskip

The following  lemma will be used in Section  \ref{effectivesec}.

 \begin{lemma}\label{corkeytris} Let $P$ be an element of $\R^N$, $\d <1$. For any $x \in \R^N$ there are two open sets $U_x$, $ W_x$ with $x \in  \overline{ U_x }\cap
  \overline{ W_x}$ such that
 \begin{eqnarray}
  \label{tris0} S_1(x,y) + P \cdot (x-y) & < 1 + |P| \, r_V(\d,1)& \txt{for any $y \in U_x$}  \\
  \label{tris00}  S_1(y,x) + P \cdot (y-x) & < 1  + |P| \, r_V(\d,1) & \txt{for any $y \in  W_x$}.
 \end{eqnarray}
 In addition both families $\{U_x\}$, $\{ W_x\}$, for $x \in \R^N$,   make up an open covering of the whole space.
\end{lemma}
\begin{proof}
Given $x \in \R^N$,  we have   by  Lemma
\ref{corkey} that
\begin{equation}\label{tris1}
     F_\d(x) \subset  F(y) \txt{for any $y \in
     \B  (x, r_V(\d,1))$}.
\end{equation}
 We set
 \begin{eqnarray*}
   U_x &=& (x + \Int F_\d(x) ) \cap \B  (x, r_V(\d,1)) \\
   W_x &=& (x - \Int F_\d(x) ) \cap  \B  (x, r_V(\d,1)),
 \end{eqnarray*}
if $y \in U_x$ then $y- x \in  \Int F_\d(x)$ and $x + t \, (y -x) \in  \B  (x, r_V(\d,1))$, for $t \in [0,1]$, we therefore  derive  from \eqref{tris1} and Lemma \ref{co}
\[ \s(x+ t \, (y-x), y-x)  < 1  \txt{for $t \in [0,1]$.}\]
 This implies that the intrinsic length of the segment
$x + t \, (y-x)$, $t \in [0,1]$ is less than   $1$ showing \eqref{tris0}.
Item \eqref{tris00} can be proven  following the same lines.

We proceed proving  that for   any $x \in \R^N$  there is
$y$ with   $ x \in  U_y$. We select $\b >0$ with
\[ \b < \min \left \{ 1, \, \frac {1-\d}{M_V L_V} \right \},\]
and  set
\[y= x - \b \, V(x).\]
We get
\[|y-x| =  \b \, |V(x)| \leq \b \, M_V < \frac {1-\d}{ L_V}=r_V(\d,1)\]
and
\[|x - y - \b \, V(y)|= \b \, |V(x)-V(y)|< \b.\]
The above inequalities show  that $x \in U_y$. Adapting the above argument, we also find $z$ with $x \in  W_z$.
This concludes the proof.
\end{proof}

\smallskip

\begin{remark}\label{disto}  Slightly adapting the above construction, a distance, denoted by  $S_a$, $\ov S_a$ can be defined starting from the sublevels $\{p \mid H(x,p) \leq a\}$, for any $a >0$. Due to the positive homogeneity of $H$ in $p$, we have for any $x$, $y$ in $\R^N$
\begin{eqnarray*}
  S_a(x,y) &=& a \, S_1(x,y) \\
  \ov S_a(x,y) &=& a \, \ov S_1(x,y)
\end{eqnarray*}
Furthermore, it is easy to check that the intrinsic  periodic distance corresponding to the level $0$ is identically vanishing, thanks to  Theorem \ref{divergence}, while $\ov S_a$  is   identically equal to $- \infty$ for $a <0$.
\end{remark}

\bigskip

\section{Homogenization}\label{effectivesec}

We consider  the family of equations

\begin{equation}\label{hjp}
    H(x,Du +P)=a \quad a \in \R,
\end{equation}
with $P \in \R^N$ fixed, and  define  the effective Hamiltonian $\ov
H$ as
\begin{equation}\label{defeffe}
    \ov H(P)= \inf\{a \mid \hbox{\eqref{hjp} admits an
usc bounded periodic subsolution}\}.
\end{equation}

 The effective Hamiltonian is the one appearing  in the limit equation \eqref{HJ} of the homogenization procedure.

\smallskip

\begin{lemma} The effective Hamiltonian $\ov H$ is convex positively homogeneous  and, in addition
\[ \min_{x \in \R^N} H(x,P) \leq \ov H(P) \leq  \max_{x \in \R^N}
H(x,P)   \txt{for any $P \in \R^N$}\]
\end{lemma}
\begin{proof}
The effective Hamiltonian $\ov H$  directly inherits   convexity and positive homogeneity from  $H$. If $a \geq \max_{x \in \R^N} H(x,P)$ then any constant function  is a
subsolution to \eqref{hjp}, which gives the rightmost  inequality in
the statement.

Next, let $u$ be a  bounded periodic usc subsolution  to \eqref{hjp}.
  We   denote by $x_0$ a maximizer of it
in $\R^N$. Then $H(x_0,P) \leq a$ and so $\min_{x \in \R^N} H(x,P)
\leq a$. This completes the proof.
\end{proof}

\medskip

We proceed fixing $P$. Following the same procedure as in Section \ref{distone}, we see that  the intrinsic distance related to the sublevel $H(x, p +P) \leq a$ is
\begin{equation}\label{metro2}
   D_a(x,y) = S_a(x,y) + P \cdot (x-y),
\end{equation}
and the periodic version is
\begin{eqnarray}
  \ov D_a(x,y) &=& \inf_{z, \, z' \in \Z^N}\{S_a(x+z,y+z') + P \cdot (x+ z-y-z')\}   \label{metro3}\\
  &=& \inf_{z \in \Z^N}\{S_a(x,y+z) + P \cdot (x -y-z)\} \nonumber
\end{eqnarray}

\begin{remark} Note that $\ov D_a$ could be finite even for $a <0$, despite the fact that $\ov S_a$ is identically equal to $- \infty$. This is related to the problem  of the possible coercivity of $\ov H$. We do not treat this issue here, see \cite{CNS}, \cite{XY}.
\end{remark}

 The next result is a relevant consequence of Theorem  \ref{divergence}.

\begin{theorem}\label{bounded}   The distance $ \ov D_1$  is
 bounded  from above in $\R^N \times \R^N$.
\end{theorem}
\begin{proof}  We first show that $\ov D_1$ is bounded from above in $\R^N \times \R^N$, provided it is finite.
 Given $\d <1$, we consider the families of open sets $\{U_x\}$,
$\{W_x\}$, for $x \in \ov Q_1$, introduced in Lemma
\ref{corkeytris}, and extract finite subfamilies    $\{U_i\}$, $i=1,
\cdots,M_1$, $\{W_j\}$, $j=1, \cdots, M_2$, respectively,
corresponding to points $x_i$, $\ov x_j$, covering $\ov Q_1$.

 Due to  Lemma \ref{corkeytris}, we have that for any $x \in \ov Q_1$
there are indices $i_0$, $j_0$ with
\begin{equation}\label{again007}
  \ov  D_1(x_{i_0},x) \leq 1 + |P| \, r_V(\d,1) \qquad\hbox{and} \qquad \ov D_1(x,\ov x_{j_0}) \leq 1 + |P| \, r_V(\d,1).
\end{equation}

Since we are supposing that $\ov D_1$ is finite, we have
\[ \ov  D_1(\ov x_j,x_i) \leq \a \txt{ for $i=1, \cdots,M_1$, $j=1,
\cdots,M_2$, some $\a >0$.}\] We take any pair of points $y$, $z$ in
$\ov Q_1$ and assume that \eqref{again007} holds true
with $z$ and $y$ in place of $x$, respectively. Then
\[ \ov D_1(y,z)  \leq  \ov D_1(y,\ov x_{j_0}) + \ov D_1(\ov x_{j_0}, x_{i_0}) +
\ov  D_1(x_{i_0},z) \leq 2 +  2 \, |P| \, r_V(\d,1) +\a,\]
showing the claim.

We proceed proving that since periodic  invariant sets for $F$ do  not exist in force of Theorem \ref{divergence},  then  $ \ov D_1$ is actually finite in $\R^N \times \R^N$.   It is  sufficient to show that, for any given   $y_0 \in \R^N$,  the set
\[ \O=\{ x \in \R^N \mid \ov D_1(y_0,x) < + \infty\}\]
is   periodic invariant for $F$. This in fact implies  that $\O= \R^N$, from which we in turn deduce, by the arbitrariness of $y$, that $ \ov D_1$ is finite.

It is clear by its very definition that $\O$ is periodic. To show the invariance, we take  $x_0 \in
\O$ and consider  an integral curve $\xi$  of $F$ defined in $[0,T)$ , for some $T \in \R^+ \cup \{+ \infty\}$,  issued from $x_0$. Then we have by Lemma \ref{co}
\begin{eqnarray*}
  \ov D_1(y_0,\xi(t))  &\leq& \ov D_1(y_0,x_0) + \ell_V \left (\xi\big |_{[0,t]} \right ) + P \cdot (x_0 - \xi(t)) \\
   &\leq&  \ov  D_1(y_0,x_0) + t + P \cdot (x_0 - \xi(t)) \txt{for any $t \in [0,T)$.}
\end{eqnarray*}
Therefore
the entire  curve $\xi$ is contained in $\O$, as claimed.
\end{proof}

\medskip

We proceed defining  a sequence  of coercive  Hamiltonians approximating $H$. We set for $k \in \N$
\[H_k(x,p) = \max \{H(x,p), -k \} + f_k(p)\]
where
\[f_k(p)= \left \{\begin{array}{cc}
  0 & \txt{if $|p| < k$} \\
  |p| - k & \txt{if $|p| \geq k$} \\
\end{array} \right . \]
and define the related intrinsic distances $S^k_a$, $\ov S^k_a$, $D_a^k$, $\ov D_a^k$ adapting formulae \eqref{metro0}, \eqref{metro1}, \eqref{metro2}, \eqref{metro3}.

The Hamiltonians  $H_k$ are  continuous,
periodic in the  state   and  convex in the momentum variable. They in addition satisfy:
\begin{itemize}
\item[--] the sequence $H_k(x,p)$ is nonincreasing for any $(x,p)$;
\item[--]  $H_k$ converges to $H$ locally uniformly in $\R^N \times
\R^N$ by Dini's monotone convergence theorem.
\end{itemize}
The previous properties clearly imply that
\[ H(x,p) \leq H_k(x,p) \qquad\hbox{ for any $k$, $x$, $p$.}\]
 The effective Hamiltonians
$\ov H_k$, corresponding to $H_k$, are defined as in
\eqref{defeffe} with obvious adaptations. Notice that, being $H_k$
coercive, any periodic  subsolution  to \eqref{hjp}, with $H_k$ in place of
$H$, is automatically Lipschitz--continuous. In addition the critical equation $H_k(x,P+Du) = \ov H_k(P)$ possess solutions.  We also have  that the $\ov
H_k$ are convex coercive.

\medskip

The main result of the section is:

\smallskip

\begin{theorem}\label{effectiveness} \hfill
\begin{itemize}
  \item[--] For any $P \in \R^N$, $\lim_k \ov H_k(P)= \ov H(P)$;
  \item[--] there exist  bounded periodic usc subsolutions and lsc  supersolutions to the equation $H(x,Du+P)= \ov H(P)$.
\end{itemize}
 \end{theorem}

 \smallskip

Note that, in  absence of coerciveness of $H$,  the convergence
of  $\ov H_k$ to  $\ov H$  is not a straightforward consequence of the fact that $H_k \to H$ locally uniformly.

 One half of the previous
theorem is actually easy and is indeed a direct consequence of the properties of $H_k$.  In fact, if $u$ is a
subsolution to $H_k(x,Du + P)=a$, for some a, then it enjoys the
same property with respect to $H_j(x,Du + P)=a$ , whenever $j <
k$,  and with respect to \eqref{hjp}. This actually shows:

\begin{lemma}\label{ovhkh}    The sequence $\ov H_k(P)$
is nonincreasing
 and\begin{equation}\label{ovhkh1}
    \ov H_k(P) \geq \ov H(P) \txt{ for any $k$.}
\end{equation}
\end{lemma}

\medskip

\begin{remark}\label{postovhkh1}
Since the matter is somehow intricate, let us summarize the
different monotonicity arising in the interplay between the
$H_k$ and $H$.
\begin{enumerate}
\item[--] $H_k(x,p)$ is  nonincreasing and $H_k(x,p) \to H(x,p)
$, for any $(x,p)$;
\item[--] $S^k_a(x,y)$ is nondecreasing
and $S^k_a(x,y) \leq  S_a(x,y)$, for any $x$, $y$, $a$;
\item[--] $\ov H_k(P)$ is  nonincreasing and the possible convergence
of $\ov H_k(P)$ to $\ov H(P)$ is the subject we are presently
investigating.
\end{enumerate}
\end{remark}

\bigskip

We need some  preliminary material for  the proof
of Theorem \ref{effectiveness}.
\smallskip

\begin{lemma}\label{via} Let $a \geq \ov H_1(P)$,  then
$\ov D_a$ is bounded in $\R^N \times \R^N$.
  \end{lemma}
\begin{proof}  We start proving that it is bounded from above. If $a \leq 0$ then $\ov D_a \leq \ov D_1$, if $a >0$ then
\[\ov D_a(x,y)= a \, \inf_{z \in \Z^N} \{ S_1(x,y) + P/a \cdot (x-y)\}.\]
It is then enough to prove that $\ov D_1$ is bounded from above. This property actually comes  from Theorem \ref{bounded}. To see that $\ov D_a$ is also bounded from below, we recall that there is a  periodic solution $u$ of
\[H_1(x,Dv+P)= \ov H_1(P).\]
Then $u$ is a subsolution to
\[H(x,Dv+P)= a\]
and, since $u$ is Lipschitz continuous for the coerciveness of $H_1$, we have
\[ u(x) - u(y) \leq \ov D_a(y,x) \txt{for any $x$, $y$}\]
showing the claim.
\end{proof}

\smallskip

We derive:

\begin{lemma}\label{effectiveness0}  The  family of the distances  $ \big (\ov D_{\ov H_k(P)}^k \big )_k$ are equibounded in $\R^N \times \R^N$. \end{lemma}
\begin{proof}
We start proving the boundedness from above.   Taking into account Remark \ref{postovhkh1}, we have
\[\ov D_{\ov H_k(P)}^k(x,y) \leq \ov D^k_{\ov H_1(P)}(x,y)\leq \ov D_{\ov H_1(P)} (x,y). \]
To get the converse estimate, we first notice that
\begin{equation}\label{eff1}
  0= \ov D_{\bar H_k(P)}^k(x,x)= \inf_{z \in \Z^N} \{S^k_{\ov H_k(P)}(x,x+z)- P \cdot z\}
\end{equation}
for any $x \in \R^N$. We consider a pair of points $x$, $y$ in $\R^N$. We select $\bar z \in \Z^N$  in such a way that
\[  S_{\ov H_1(P)}(y, x+ \bar z) + P \cdot (y - x - \bar z) \leq \ov D_{\ov H_1(P)}(y,x) + 1.\]
We compute taking into account \eqref{eff1}, Remark \ref{postovhkh1} and Lemma \ref{via}
\begin{eqnarray*}
 \ov  D_{\ov H(P)}^k(x,y) &=& \inf_{z \in \Z^N} \{S^k_{\ov H_k(P)}(x+z,y))+ P \cdot (x+z -y)\} \\  &\geq& \inf_{z \in \Z^N} \{S^k_{\ov H_k(P)}(x +z ,x+ \bar z)- S^k_{\ov H_k(P)}(y, x + \bar z) \\ &+& P \cdot (  z- \bar z) - P \cdot (y - x - \bar z) \} \\ & \geq& \inf_{z \in \Z^N} \{S^k_{\ov H_k(P)}(x+  z,x+ \bar z)+ P \cdot ( z- \bar  z)\}\\ & +& P\cdot (x+ \bar z -y)  - S^k_{\ov H_k(P)}(y, x + \bar z) \\ & \geq&  P\cdot (x+ \bar z -y) - S_{\ov H_k(P)} (y,x+\bar z) \\ &\geq& P\cdot (x+ \bar z -y) - S_{\ov H_1(P)} (y,x+\bar z) \geq - \ov D_{\ov H_1(P)} (y,x) -1 .
\end{eqnarray*}
This concludes the proof.
\end{proof}

\medskip

\begin{lemma}\label{ovh} If for a given $a$  there is a bounded periodic usc subsolution and a bounded  periodic lsc supersolution of \eqref{hjp},   then $a = \ov H(P)$.
\end{lemma}

The
proof is basically the same as in coercive case,  with some more
precaution. We present it in the Appendix.

\medskip

 We  need the following result, see for instance \cite{FS}.

\begin{proposition}\label{kappone}  For any $k$,  any $y$ in the Aubry set related to  $H_k$,   the function
$x \mapsto \ov D_{\ov H_k(P)}^k(y,x)$ is a periodic Lipschitz continuous solution of $H_k(x, P + Du)= \ov H_k(P)$.
\end{proposition}

\medskip

\noindent \begin{proof}[ Proof of Theorem  \ref{effectiveness}]  We know
from Proposition \ref{kappone},  that we can select, for any $k$,
$y_k$ in such a way that $v_k := \ov D_{\ov H_k(P)}^k(y_k,\cdot)$ is solution to
\[ H_k(x, Du+ P)= \ov H_k(P).\]
 Since
the $v_k$ are periodic, equibounded  by Lemma \ref{effectiveness0}, the
functions
\[v := \liminfs v_k   \qquad \hbox{and} \qquad  w:= \limsups v_k\]
are  bounded periodic and, taking into account the local uniform convergence of $H_k$ to $H$ plus  standard
stability properties of viscosity solutions theory, are indeed
lsc supersolution and usc subsolution, respectively, to
\[ H(x,Du + P)= \lim_k \ov H_k(P).\]
This implies by  Lemma \ref{ovh} that $\ov H(P) = \lim_k \ov H_k(P)$. The
assertion  then follows.

\end{proof}

\smallskip

Exploiting Theorem \ref{effectiveness}, we get the homogenization result in a quite standard way. We state it below and provide the proof in the Appendix.

\smallskip

\begin{theorem} \label{homohomo} The solutions to \eqref{HJe} locally uniformly converge to the solution of \eqref{HJ}, where $\ov H$ is the effective Hamiltonian defined in \eqref{defeffe}.

\end{theorem}

\appendix

\section{Some Proofs}

\noindent\begin{proof}[Proof of Proposition \ref{sigmalsc}] Let us consider $(x_0,q_0)$ and an approximating  sequence
$(x_n,q_n)$. First assume $\s(x_0,q_0) < + \infty$
and pick $p_0 \in Z(x_0)$ with $ \s(x_0,q_0)  \leq p_0 \cdot q_0 +
\d$, for some positive $\d$. Taking into account that the interior of $Z(x_n)$ is nonempty,  there is a
sequence $p_n \in Z(x_n)$  converging to $p_0$. We find
\[ \liminf_n \s(x_n,q_n) \geq \lim_n p_n \cdot q_n= p_0 \cdot q_0
\geq \s(x_0,q_0) - \d,\] and the claimed semicontinuity follows
since $\d$ has been arbitrarily chosen .

If $\s(x_0,q_0) =  + \infty$, then taking  any $\b  >0$,  we find
$p_0 \in Z(x_0)$ with $p_0 \cdot q_0 \geq \b$, then
\[ \liminf_n \s(x_n,q_n) \geq \lim_n p_0 \cdot q_n= p_0 \cdot q_0
\geq \b,\] which  implies $\lim_n \s(x_n,q_n)= + \infty$.

The claimed positive homogeneity and subadditivity are immediate.

\end{proof}

\medskip
\noindent \begin{proof}[Proof of Lemma \ref{ovh}]  The argument is by contradiction.  Let  $a$ be such that \eqref{hjp} has
  supersolution and subsolution as specified in the statement, denoted by  $v$, $u$, respectively.  If $ a > \ov H(P)$, then  we can select $b < a$ with
 \[H(x,Du + P) \leq b \txt{ in
the viscosity sense.}\] We denote, for any $\d
>0$, by $u^\d$ the  sup--convolutions  of $u$. By basic properties of it, we know that
\[H \left (y, \frac{x - y}\d  \right )\leq b \]
for any $\d$, any $x \in \R^N$, $y$ $u^\d$--optimal to $x$.
We proceed considering a sequence
\begin{equation}\label{ovh0}
   x_\d \in \arg\min_{\R^N} \{v - u^\d\},
\end{equation}
note that such minimizers do exist because $v$ is lsc, $u^\d$ usc, and both $v$  and $u^\d$ are periodic.  We denote,
for any $\d$, by $y_\d$ an element $u^\d$--optimal for $x_\d$, we have
 \begin{equation}\label{ovh1}
   \lim_\d   \frac{|x_\d - y_\d|^2}\d =0.
\end{equation}
Consequently, taking into account the inequality
\[\left |H \left (x_\d,\frac{x_\d - y_\d}\d \right )- H \left (y_\d,\frac{x_\d - y_\d}\d \right ) \right | \leq L_V \, \frac{|x_\d-y_\d|^2}\d\]
and that $b <a$, we conclude that
\begin{equation}\label{ovh2}
    H \left (x_\d, \frac{x_\d - y_\d}\d  \right ) < a,
\end{equation}
for $\d$ small enough. On the other side,  keeping in mind that
$\frac{x_\d - y_\d}\d \in D^-u^\d(x_\d) $ and that  $u^\d$ is
subtangent to $v$ at $x_\d$, we get
\[ H \left (x_\d, \frac{x_\d - y_\d}\d  \right )\geq a \quad\txt{for any $\d$,} \]
which contradicts   \eqref{ovh2}.

\end{proof}

\medskip

\noindent\begin{proof}[Proof of Theorem \ref{homohomo}]  Let $u_\eps$ the solutions to \eqref{HJe}, they are equibounded by {\bf (A3)}. We set
\[v = \liminfs u_\eps   \qquad \hbox{and} \qquad  u= \limsups u_\eps.\]
As already pointed out in the Introduction, it is standard to show that $v$ is a supersolution of \eqref{HJ}.

We focus on proving  that $u$ is subsolution to \eqref{HJ}.
 We consider $(x_0,t_0) \in \R^N \times (0,+ \infty)$ and  a $C^1$ strict  supertangent $\psi$ to $u$ at $(x_0,t_0)$. We denote by
   $w$   a bounded periodic lsc supersolution to the  problem
\[H(x,Dv + D(\psi(x_0,t_0)) )= \ov H(D(\psi(x_0,t_0))).\]
We have that  $\eps \, w(x /\eps)$  is supersolution to
\begin{equation}\label{homo01}
 H(x/\eps,Dv + D(\psi(x_0,t_0)))= \ov H(D(\psi(x_0,t_0))).
 \end{equation}
  The point  $(x_0,t_0)$
is the unique maximizer of $u - \psi$ in $\ov U$, for a suitable
choice of an open  neighborhood $U$ of $(x_0,t_0)$ in $\R^N \times
(0,+ \infty)$. We proceed considering the functions
\begin{equation}\label{homo1}
 u_\eps(x,t) - \psi(x,t) - \eps \, w_{\delta}(x/\eps),
\end{equation}
where  $w_{\delta}$ stands for the $\delta$ inf--convolution of $w$ for $\delta >0$.
 We have that
 the quantity $\eps \, w_\delta(x/\eps)$ converges uniformly to
 $0$ as $\eps \to 0$, uniformly with respect to $\delta$. Owing to the stability properties of
 maximizers, we deduce
 that  a sequence of maximizers  of \eqref{homo1} denoted by $(x^\delta_\eps,t^\delta_\eps)$, belongs to
 $U$ for $\eps$ small enough and
 \begin{equation}\label{homo10}
   (x^\delta_\eps,t^\delta_\eps) \to (x_0,t_0) \qquad\hbox{as $\eps \to 0$, uniformly in $\d$.}
\end{equation}
Since
 \[ (\psi_t(x^\delta_\eps,t^\delta_\eps), D\psi(x^\delta_\eps,t^\delta_\eps) +  p^\delta_\eps) \in D^+
 u_\eps(x^\delta_\eps,t^\delta_\eps)\]
for any $p^\delta_\eps \in \partial w_{\delta}(x^\delta_\eps) =
 D^+w_\delta(x^\delta_\eps)$,  and $u_\eps$ is subsolution to \eqref{HJe}, we
 derive
 \begin{equation}\label{homo2}
  \psi_t(x^\delta_\eps,t^\delta_\eps)+ H(x^\delta_\eps/\eps, D\psi(x^\delta_\eps,t^\delta_\eps)+p^\delta_\eps) \leq
 0  \quad\hbox{for any $p^\delta_\eps \in \partial  w_\delta(x^\delta_\eps)$.}
\end{equation}
 Since $w_\delta$  is inf--convolution of $w$, and $w$ is supersolution to the equation \eqref{homo01}, we further  have
 \[H(y^\delta/\eps, Dw_\delta(x/\eps) + D\psi(x_0,t_0)) \geq \ov
 H(D\psi(x_0,t_0))\]
 at any  $x$ where $w_\delta$ is differentiable, with $y^\delta$
 denoting the $w_\delta
$--optimal point for $x$. Consequently
 \begin{equation}\label{homo3}
 H(y^\delta_\eps/\eps, q^\delta_\eps + D\psi(x_0,t_0)) \geq \ov H(D\psi(x_0,t_0))
\end{equation}
for any $\eps$, some $q^\delta_\eps \in \partial w_\delta(x^\delta_\eps)$, with $y^\delta_\eps$ denoting an  $u_\delta$--optimal point for $x^\delta_\eps$.  In order to combine  \eqref{homo2} and  \eqref{homo3},  we need to estimate
\begin{equation}\label{homo5}
   |H(x^\delta_\eps/\eps, D\psi(x^\delta_\eps,t_\eps) +  q^\delta_\eps) - H(y^\delta_\eps/\eps, D\psi(x_0,t_0) + q^\delta_\eps)|=:I.
\end{equation}
After straightforward computations, we get
\begin{eqnarray*}
 I &\leq& |D\psi(x_0,t_0) - D\psi(x^\delta_\eps, t^\delta_\eps)| \, (1 + |V(x^\delta_\eps/\eps)|) \\
   &+&  |D\psi(x_0,t_0) + q_\eps^\delta| \, L_V \,  \frac {|x^\delta_\eps - y^\delta_\eps|}\eps \\
   &\leq&\OO(\eps) + \frac{\oo(\sqrt\d)}\eps + \frac{\oo(\d)}\d \, \frac 1\eps.
\end{eqnarray*}
We freeze $\eps$ and move $\d$ in order to obtain a $\d_\eps$ for which
\[I \leq \OO(\eps) + \frac{\eps^2}\eps +  \frac {\eps^2}\eps  =\OO(\eps).\]
We use the above estimate plus \eqref{homo2},  \eqref{homo3} to finally get
\begin{eqnarray*}
  0 &\geq& \psi_t(x^{\delta_\eps}_\eps/,t^{\delta_\eps}_\eps)+ H(x^{\delta_\eps}_\eps/\eps, D\psi(x^{\delta_\eps}_\eps,t^{\delta_\eps}_\eps)+q^{\delta_\eps}_\eps) \\
  &\geq& \psi_t(x^{\delta_\eps}_\eps,t^{\delta_\eps}_\eps) +
  H(y^{\d_\eps}_\eps/\eps, D\psi(x_0,t_0)+q^{\delta_\eps}_\eps) - \OO(\eps)\\
  &\geq& \psi_t(x^{\delta_\eps}_\eps,t^{\delta_\eps}_\eps) + \ov H(D\psi(x_0,t_0)) - \OO(\eps).
\end{eqnarray*}
Sending $\eps \to 0$, we see that $u$ satisfies the subsolution test at $(x_0,t_0)$ with respect to the supertangent $\psi$, as was claimed.

Finally, taking into account that \eqref{HJ} satisfies a comparison principle, we get that the $u_\eps$ locally uniformly converge to $v=u$.

\end{proof}

\end{document}